\documentclass[12pt]{amsart}
\usepackage{amsmath,amssymb}
\usepackage[colorlinks=true,linkcolor=black,citecolor=black]{hyperref}
\usepackage{geometry}

\newcommand{\leqs}{\leqslant }
\newcommand{\geqs}{\geqslant }

\newtheorem{lemma}{Lemma}
\newtheorem{proposition}{Proposition}
\newtheorem{theorem}{Theorem}

\begin{document}

\title[Upper bounds for moments of the Riemann zeta-function]{Sharp upper bounds for fractional moments of the Riemann zeta function}

\author{Winston Heap}
\address{Department of Mathematics, University College London, 25 Gordon Street, London WC1H.}
\email{winstonheap@gmail.com}

\author{Maksym Radziwi\l\l}
\address{Department of Mathematics
  Caltech, 1200 E California Blvd
  Pasadena, CA, 91125}
\email{maksym.radziwill@gmail.com}

\author{K. Soundararajan}
\address{Department of Mathematics, Stanford University, Stanford, CA 94305, USA.}
\email{ksound@stanford.edu}

\thanks{The first author is supported by European Research Council grant no.~ 670239. The second author acknowledges the support of a Sloan fellowship.   The third author is partially supported by a grant from the National Science 
Foundation, and by a Simons Investigator grant from the Simons Foundation}

\maketitle

\begin{abstract}
  We establish sharp upper bounds for the $2k$th moment of the Riemann zeta function on the critical line, for all real $0 \leqs k \leqs 2$.
  This improves on earlier work of Ramachandra, Heath-Brown and Bettin-Chandee-Radziwi{\l\l}. 
  \end{abstract}

\section{Introduction}

This paper is concerned with the moments of the Riemann zeta function on the critical line: namely, with the quantity 
$$
I_k(T)=\int_{T}^{2T}| \zeta(\tfrac{1}{2}+it)|^{2k} dt, 
$$
where $k> 0$ is real and $T$ is large.   The problem of understanding the behavior of these moments 
is central in the theory of the Riemann zeta-function.  The classical work of Hardy and Littlewood \cite{HL}, and Ingham \cite{I} established 
asymptotic formulae for $I_k(T)$ in the cases $k=1$ and $2$, and these still remain the only situations where an asymptotic 
is known.   Lacking an asymptotic, much work has been focussed on the problems of obtaining sharp upper and lower bounds for these moments.  
 Lower bounds of the form $I_k(T) \gg_{k} T (\log T)^{k^2}$ are established for all $k \geqs 1$ in Radziwi{\l \l} and Soundararajan \cite{RS lower} unconditionally, 
 and for all $k \geqs 0$ conditionally on the Riemann Hypothesis in papers of Heath-Brown and Ramachandra, see \cite{Ram1, Ram2, HB}.   
 Upper bounds of the form $I_{k}(T) \ll_{k} T (\log T)^{k^2}$ are known when $k = 1/n$ for natural numbers $n$ (due to Heath-Brown \cite{HB}) and 
 when $k=1+1/n$ for natural numbers $n$ (by work of Bettin, Chandee, and Radziwi{\l \l} \cite{BCR}).    
 Conditionally on the Riemann Hypothesis, the work of Harper \cite{Ha}, refining earlier work of Soundararajan \cite{S}, establishes that $I_k(T) \ll_{k} T (\log T)^{k^2}$ 
 for all $k\geqs 0$.   This paper adds to our knowledge on moments by establishing a 
sharp upper bound for $I_k(T)$ for all real $0 \leqs k \leqs 2$.

\begin{theorem} \label{main theorem}
  Let $0 \leqs k \leqs 2$. Then, for $T \geqs 10$, 
  $$
  I_k(T) \ll T (\log T)^{k^2}.
  $$
\end{theorem}

The proof of the theorem is based on the method introduced in Radziwi{\l \l} and Soundararajan \cite{RS1} which  enunciates that 
if in a family of $L$-values, asymptotics for a particular moment can be established with a little room to spare, then sharp upper bounds 
may be obtained for all smaller moments.   Theorem \ref{main theorem} is an illustration of this principle, and combines the ideas of 
\cite{RS1} together with knowledge of the fourth moment of $\zeta(s)$ twisted by short Dirichlet polynomials (see the work of Hughes and Young 
\cite{HY}, and Betin, Bui, Li, and Radziwi{\l \l} \cite{BBLR}).

\section{Plan of the Proof of Theorem \ref{main theorem}} \label{se:proof}

\noindent Throughout, $\log_j$ will denote the $j$-fold iterated logarithm.   Let $T$ be large, and let 
$\ell$ denote the largest integer such that $\log_{\ell} T \geqs 10^4$.   Define a sequence $T_j$ by setting 
$T_1 =e^2$, and for $2\leqs j \leqs \ell$ by 
$$ 
 T_{j} := \exp \Big ( \frac{\log T}{(\log_{j} T)^2} \Big ).  
$$
Note that $e^2 =T_1 < T_2 < \ldots < T_\ell \leqs T^{10^{-8}}$.

For each $2 \leqs j \leqs \ell$, set
$$
\mathcal{P}_{j}(s) := \sum_{T_{j - 1} \leqs p < T_j} \frac{1}{p^s}, \qquad \text{ and } \qquad P_j = {\mathcal P}_j(1) = \sum_{T_{j-1} \leqs p < T_j} \frac 1p.  
$$
Note that for large $T$, 
$$
P_j \sim \log \frac{\log T_{j}}{\log T_{j-1}} =2 \log \Big(\frac{\log_{j-1} T}{\log_j T}\Big) = 2\log_j T - 2\log_{j+1} T,
$$ 
so that $P_\ell \geqs 10^4$, $P_{\ell-1} \geqs \exp(10^4)$, and so on. 
Further, define 
\begin{equation} \label{eq:nk}
\mathcal{N}_{j}(s; \alpha) := \sum_{\substack{p | n \implies T_{j - 1} \leqs p < T_j \\ \Omega(n) \leqs 500 P_j}} \frac{\alpha^{\Omega(n)} g(n)}{n^s} 
\end{equation}
where $g(n)$ denotes the multiplicative function given on prime powers by $g(p^{r}) = 1/r!$.

 The motivation for these definitions is the following.  Typically one might expect that $\zeta(\frac 12+it)^{\alpha}$ is similar to $\prod_{j\leqs \ell} \exp(\alpha {\mathcal P}_j(\frac 12+it))$.  
 Now most of the time, $|{\mathcal P}_j(\tfrac 12+ it)|$ is no more than $50 P_j$, in which case by a Taylor approximation one can approximate 
 $\exp(\alpha {\mathcal P}_j(\frac 12+it))$ by ${\mathcal N}_j(\frac12+it;\alpha)$ (see Lemma 1 below).   Thus, for most $t$ we shall be able to replace $\zeta(\frac 12+it)^{\alpha}$  
by $\prod_{j\leqs \ell} {\mathcal N}_{j} (\frac 12+it;\alpha)$, which is a short Dirichlet  polynomial (of length $\leqs T^{1/10}$, say) and thus facilitates computations.

We now state three propositions from which the main theorem will follow, postponing the proofs of the propositions to later sections.  

\begin{proposition}\label{inequality proposition}
  Let $0 \leqs k \leqs 2$ be a given real number. 
  Then, for all complex numbers $s$ inside the critical strip $0 < \text{Re } s < 1$,
  \begin{align*}
    |\zeta(s)|^{2k} & \leqs k |\zeta(s)|^4 \prod_{2\leqs j \leqs \ell} |\mathcal{N}_j(s; k - 2)|^2 + (2 - k) \prod_{2\leqs j \leqs \ell} |\mathcal{N}_j(s; k)|^2
                      \\ & + \sum_{2\leqs v \leqs \ell} \Big ( k  |\zeta(s)|^4 \prod_{2\leqs j < v} |\mathcal{N}_j(s, k - 2)|^2 + (2 - k) \prod_{2\leqs j < v} |\mathcal{N}_j(s; k)|^2 \Big )  \Big | \frac{\mathcal{P}_{v}(s)}{50 P_v} \Big |^{2 \lceil 50 P_v \rceil}.  
  \end{align*}
\end{proposition}


\begin{proposition} \label{prop1}
  Let $0 \leqs k \leqs 2$ real, be given. Then 
$$ 
\int_T^{2T} \prod_{2\leqs j\leqs \ell} |{\mathcal N}_j(\tfrac 12+it;k)|^2 dt \ll T(\log T)^{k^2}, 
$$ 
and for all $2 \leqs v \leqs \ell$ and $0 \leqs r \leqs \lceil 50 P_v \rceil$, 
$$
 \int_{T}^{2T} \prod_{2\leqs j < v} |\mathcal{N}_j(\tfrac 12 + it; k) |^2 |\mathcal{P}_{v}(\tfrac 12 + it)|^{2 r} d t 
 \ll T  (\log T_{v-1})^{k^2}  \big( r! P_v^{r}\big). 
  $$
\end{proposition}
\begin{proposition} \label{prop2}
  Let $0 \leqs k \leqs 2$ real, be given. Then 
  $$ 
  \int_T^{2T} |\zeta(\tfrac 12+ it)|^4 \prod_{2\leqs j\leqs \ell} |{\mathcal N}_j(\tfrac 12+it; k-2)|^2 dt \ll T (\log T)^{k^2}, 
  $$ 
  and for all $2 \leqs v \leqs \ell$ and $0 \leqs r \leqs \lceil 50 P_v \rceil$,
  \begin{align*}
  \int_{T}^{2T}  |\zeta(\tfrac 12 + it)|^4 & \prod_{2\leqs j < v} |\mathcal{N}_j ( \tfrac 12 + it; k - 2) |^2
   |\mathcal{P}_v(\tfrac 12 + it)|^{2r} dt \\
   & \ll T (\log T)^{4} (\log T_{v-1})^{k^2-4} \Big( 18^r r! P_v^r \exp(P_v)\Big). 
  \end{align*}
\end{proposition}

We quickly deduce Theorem \ref{main theorem} from the above propositions.

\begin{proof}[Proof of Theorem \ref{main theorem}]  
  Combining the above propositions we find
  \begin{align*} \label{eq:firstterm}
  \int_{T}^{2T}  |\zeta(\tfrac 12 + it)|^{2k} dt &\ll T(\log T)^{k^2} 
  + \sum_{2 \leqs v \leqs \ell} T(\log T_{v - 1})^{k^2}  \nonumber\\ 
  &\times\Big( \frac{\lceil 50P_v\rceil! P_v^{\lceil 50P_v\rceil}}{(50P_v)^{2\lceil 50P_v\rceil}} 
  + \Big(\frac{\log T}{\log T_{v-1}}\Big)^4 \frac{18^{\lceil 50P_v\rceil }\lceil 50P_v\rceil! P_v^{\lceil 50P_v\rceil} \exp(P_v)}{(50P_v)^{2\lceil 50P_v\rceil} } \Big).  
    \end{align*} 
   A quick calculation shows that the above is 
   \begin{align*} 
&   \ll T(\log T)^{k^2} \Big( 1+  \sum_{2\leqs v \leqs \ell}  \Big( \frac{\log T}{\log T_{v-1}}\Big)^4 \exp(-50 P_v)\Big)\\
& \ll 
   T(\log T)^{k^2} \Big( 1+ \sum_{2\leqs v\leqs \ell} (\log_{v-1} T)^8 \Big(\frac{\log_v T}{\log_{v-1} T}\Big)^{100}\Big) \ll T(\log T)^{k^2} . 
  \end{align*}  
  \end{proof}

\section{Proof of Proposition \ref{inequality proposition}}

\begin{lemma}\label{short poly lemma}
 Let $|\alpha| \leqs 2$ be a real number, $T$ be sufficiently large, and $s$ be a complex number. For all $2\leqs j\leqs \ell$, if $| \mathcal{P}_j(s) | \leqs 50 P_j$ then
 $$
 \exp (2 \alpha  \text{Re } \mathcal{P}_j (s)) \leqs \left( 1 - e^{-P_j} \right)^{-1}
 | \mathcal{N}_j (s; \alpha) |^2.
 $$
\end{lemma}

\begin{proof}  Expanding $\exp(\alpha {\mathcal P}_j(s))$ using a Taylor series, and using the assumption $|\mathcal{P}_j(s)| \leqs 50 P_j$, 
we find that 
$$ 
|\exp(\alpha {\mathcal P}_j(s))| \leqs \Big| \sum_{m\leqs 500 P_j} \frac{\alpha^m {\mathcal P}_j(s)^m}{m!} \Big| + 2 \cdot \frac{(100 P_j)^{500P_j}}{\lceil 500P_j\rceil !}  .
$$ 
The last term is $\leqs e^{-250 P_j}$, while $|\exp(\alpha {\mathcal P}_j(s))| \geqs \exp(-|\alpha| 50P_j) \geqs \exp(-100P_j)$.   Therefore, since $P_j \geqs 10^4$, we may easily 
conclude that 
  $$
  |\exp(\alpha \mathcal{P}_j(s))|^2 \leqs (1 - e^{-P_j})^{-1}  \Big | \sum_{m \leqs 500 P_j} \frac{\alpha^m \mathcal{P}_j(s)^m}{m!} \Big |^2.
  $$
  Since 
    $$
\frac{  \mathcal{P}_j(s)^{m}}{m!}  = \frac{1}{m!} \sum_{T_{j - 1} \leqs p_1, \ldots, p_m < T_j} \frac{1}{(p_1 \ldots p_m)^s} = 
\sum_{\substack{p | n \implies T_{j - 1} \leqs p < T_j \\ \Omega(n) = m}} \frac{g(n)}{n^s}, 
  $$
  the proposition follows.  
  %
\end{proof}

\begin{proof}[Proof of Proposition \ref{inequality proposition}]
This proposition is an analogue of Lemma 2 of \cite{RS1}, and is proved similarly.  We make use of Young's inequality $ab \leqs a^p/p + b^q/q$ for any non-negative 
real numbers $a$ and $b$, and non-negative $p$ and $q$ with $1/p+1/q=1$.    

If $| \mathcal{P}_j (s) | \leqs 50 P_j$ for all $2\leqs j \leqs \ell$ then using Young's inequality with $p = 4 / 2 k$ and $q = 4 / (4 - 2 k)$ we have
$$
|\zeta(s)|^{2k} \leqs \frac{k}{2}  | \zeta (s) |^{4}  \prod_{2\leqs j \leqs \ell} e^{(- 4 + 2k)
    \text{Re }\mathcal{P}_j (s)} + \Big(1 - \frac{k}{2}\Big)  \prod_{2\leqs j \leqs \ell} e^{2k\text{Re }\mathcal{P}_j
    (s) }.
$$
By Lemma \ref{short poly lemma} the right hand side is 
$$
\leqs  \prod_{2\leqs j \leqs \ell} (1 - e^{-P_j})^{-1} \Big( \frac{k}{2}  | \zeta (s) |^{4}  \prod_{2\leqs j \leqs \ell} |\mathcal{N}_j(s; k - 2)|^2 
   + \Big(1 - \frac{k}{2} \Big)  \prod_{2\leqs j \leqs \ell} |\mathcal{N}_j(s; k)|^2 \Big).  
$$
Since   $\prod_{2\leqs j\leqs \ell} (1 - e^{-P_j})^{-1} \leqs 2$, this contribution is bounded by  the first two terms in the proposition.

Now suppose that there exists an integer $2\leqs v\leqs \ell$ for which $| \mathcal{P}_j (s) | \leqs 50 P_j$ whenever $2\leqs j < v$, but with $| \mathcal{P}_{v} (s) | > 50 P_{v}$. Then applying Young's inequality and Lemma \ref{short poly lemma} as before, and noting that $| \mathcal{P}_v (s) | / (50 P_v) \geqs 1$, we find
  \begin{align*} 
  | \zeta (s) |^{2 k} & \leqs \Big ( k  | \zeta (s) |^4  \prod_{2\leqs j < v} | \mathcal{N}_j(s; k - 2) |^2  + (2-k)  \prod_{2\leqs j < v} | \mathcal{N}_j(s;k) |^2 \Big ) \Big | \frac{\mathcal{P}_v(s)}{50 P_v} \Big |^{2 \lceil 50 P_v \rceil} .
  \end{align*}
  Summing this over all $2\leqs v\leqs \ell$, we obtain Proposition \ref{inequality proposition}.
\end{proof}

\section{Proof of Proposition \ref{prop1}}

We give a proof of the second assertion of the proposition, the first statement being similar.  
Since $\prod_{2\leqs j < v} \mathcal{N}_j(s; k) \mathcal{P}_v(s)^{r}$ is a Dirichlet polynomial of length $\leqs T^{1/10}$, 
using the familiar mean value estimate for Dirichlet polynomials, we find that
\begin{align*}
\int_{T}^{2T} & \ \prod_{2\leqs j < v} |\mathcal{N}_j(s; k) |^2  |\mathcal{P}_v(s)|^{2r} d t\\
& \ll T \prod_{2\leqs j < v} \Big ( \sum_{\substack{p | n_j \implies T_{j - 1} \leqs p < T_j \\ \Omega(n_j) \leqs 500 P_j}} \frac{k^{2 \Omega(n_j)}}{n_j} \Big )
 \sum_{\substack{p|n \implies T_{v-1} \leqs p < T_v \\ \Omega(n)=r}} \frac{(r! g(n))^2}{n}.
\end{align*}
Now note that,  
$$
\sum_{\substack{p | n_j \implies T_{j - 1} \leqs p < T_{j}}} \frac{k^{2 \Omega(n_j)}}{n_j} \leqs \prod_{T_{j-1} \leqs p < T_j} 
\Big( 1+ \frac{k^2}{p} + \frac{k^4}{p^2} + \ldots \Big) \ll \Big( \frac{\log T_{j}}{\log T_{j-1}} \Big)^{k^2},
$$
where we used that $p \geqs T_1 \geqs e^2 > k^2$ so that the convergence of $\sum_{r=0}^{\infty} k^{2r}/p^r$ is assured.
Further, since $g(n) \leqs 1$ always, 
$$ 
\sum_{\substack{p|n \implies T_{v-1} \leqs p< T_v \\ \Omega(n)=r}} \frac{(r! g(n))^2}{n} \leqs r! 
\sum_{\substack{p|n \implies T_{v-1} \leqs p< T_v \\ \Omega(n)=r}} \frac{r! g(n)}{n} = r! P_v^r.
$$ 
The second assertion of the proposition follows.  

\section{Twisted  fourth moments}

In order to establish Proposition \ref{prop2} we shall require a formula for the twisted fourth moment,
$$
\int_{T}^{2T} |\zeta(\tfrac 12 + it)|^4 \cdot \Big | \sum_{n \leqs T^{\theta}}\frac{a(n)}{n^{1/2 + it}} \Big |^2 \Phi \Big ( \frac{t}{T} \Big ) dt, 
$$
where $\Phi$ is a smooth non-negative function such that $\Phi(x) \geqs 1$ for $1 \leqs x \leqs 2$.  Such mean values have been considered by many 
authors (for example  see \cite{HY}), and we shall make use of the asymptotic established in \cite{BBLR}.

To state the asymptotic formula, we introduce some notation.  Put 
$$
A_{z_1,z_2,z_3,z_4}=\frac{\zeta(1+z_1+z_3)\zeta(1+z_1+z_4)\zeta(1+z_2+z_3)\zeta(1+z_2+z_4)}
{\zeta(2+z_1+z_2+z_3+z_4)}, 
$$
and 
\begin{equation} 
\label{5.1} 
B_{z_1,z_2,z_3,z_4}(n)=\prod_{p^{n_p} \Vert n} 
\Big( \sum_{j\geqs 0} \frac{ \sigma_{z_1,z_2}(p^{n_p+j})\sigma_{z_3,z_4}(p^{j})}{p^{j}}  \Big) \Big( 
\sum_{j\geqs 0}\frac {\sigma_{z_1,z_2}(p^{j})\sigma_{z_3,z_4}(p^{j})}{p^{j}}  \Big)^{-1} 
\end{equation} 
where
$\sigma_{z_1,z_2}(n)=\sum_{n_1n_2=n}n_1^{-z_1}n_2^{-z_2}$
and $n_p$ is the highest power of $p$ dividing $n$.  Finally, define 
\begin{equation} 
\label{5.2} 
F(z_1,z_2,z_3,z_4) = A_{z_1,z_2,z_3,z_4} \sum_{m, n} \frac{a(n)\overline{a(m)}}{[m,n]} B_{z_1,z_2,z_3,z_4}\Big( \frac{n}{(m,n)}\Big) B_{z_3,z_4,z_1,z_2}\Big( \frac{m}{(m,n)}\Big). 
\end{equation} 
Note that $F$ depends on the coefficients of the Dirichlet polynomial twisting the fourth moment.  

\begin{proposition} \label{prop3}  Let $T\geqs 2$ and let $\Phi(x)$ be a smooth function supported on $[1/2,4]$ satisfying $\Phi^{(j)}(x)\ll_{\varepsilon} T^\varepsilon$ 
for any $j\geqs 0$ and all $\varepsilon > 0$. Let $a(n)$ be a sequence of complex numbers obeying the bound $|a(n)| \ll_{\varepsilon} n^{\varepsilon}$ for 
all $n \geqs 1$ and all $\varepsilon > 0$. Then, for $\theta < \tfrac 14$, we have 
  \begin{align*}
    \int_{\mathbb{R}} & |\zeta(\tfrac 12 + it)|^4 \cdot \Big | \sum_{n \leqs T^{\theta}} \frac{a(n)}{n^{1/2 + it}} \Big |^2 \Phi \Big ( \frac{t}{T} \Big ) dt = O(T^{1-\epsilon}) + \\ & \frac{1}{4(2\pi i)^4} \int_{\substack{ |z_j| = 3^j / \log T \\ 1\leqs j\leqs 4}} F(z_1, z_2, z_3, z_4) \Delta(z_1, z_2, -z_3, -z_4)^2 \Big ( \int_{\mathbb{R}} \Phi \Big ( \frac{t}{T} \Big ) \prod_{j = 1}^{4} \Big ( \frac{t}{2\pi} \Big )^{z_j/2} dt \Big ) \prod_{j = 1}^{4} \frac{d z_j}{z_j^4}
  \end{align*}
  where 
  $$ 
  \Delta(z_1,z_2, z_3,z_4) = \prod_{1\leqs i < j \leqs 4} (z_j-z_i) 
  $$ 
    denotes the Vandermonde determinant.
\end{proposition}

\begin{proof}
  Theorem 1 in \cite{BBLR} gives an asymptotic formula for
\begin{equation*} \label{shifts}
\int_{{\Bbb R}}  \zeta(\tfrac 12 + \alpha_1 + it)\zeta(\tfrac 12 + \alpha_2 + it) \zeta(\tfrac 12 + \alpha_3 - it)\zeta(\tfrac 12 + \alpha_4 - it) \Big | \sum_{n \leqs T^{\theta}} \frac{a(n)}{n^{1/2 + it}} \Big |^2 \Phi\Big(\frac tT\Big) dt,  
\end{equation*}
with $\alpha_1, \alpha_2, \alpha_3, \alpha_4$ complex numbers of modulus $\ll (\log T)^{-1}$. We apply Lemma 2.5.1 of \cite{cfkrs} to express that formula  in terms of a multiple contour integral. Setting all the shifts $\alpha_j$ equal to zero then gives the claim. 
\end{proof}

\section{Proof of Proposition \ref{prop2}}

Again we confine ourselves to proving the second assertion of the proposition; the first statement follows similarly.  
We apply Proposition \ref{prop3} with coefficients $a(n)$ given by 
$$
\sum_{n} \frac{a(n)}{n^s} = \Big ( \prod_{2\leqs j < v} \mathcal{N}_j(s; k-2) \Big ) \mathcal{P}_v(s)^{r},  
$$
and taking $\Phi$ to be a non-negative smooth function supported on $[1/2,4]$ with $\Phi(x) =1$ on $[1,2]$.  
On the circles $|z_j| = 3^j/\log T$ (for $1\leqs j\leqs 4$) we note that
$$
\Delta(z_1,z_2,-z_3,-z_4)^2 \ll (\log T)^{-12},   \ \ A_{z_1, z_2, z_3, z_4} \ll (\log T)^{4},
$$
and that
$$
 \int_{\Bbb R} \Phi\Big(\frac tT\Big) \prod_{j=1}^{4}\Big(\frac{t}{2\pi }\Big)^{z_j/2} dt \ll T. 
$$ 
Therefore by Proposition \ref{prop3} we conclude that 
\begin{equation*}\label{I_2 initial bound}
\begin{split}
     \int_{T}^{2T} |\zeta(\tfrac 12 + it)|^4  \Big | \sum_{n} \frac{a(n)}{n^{1/2 + it}} \Big |^2 dt \ll T(\log T)^4 \cdot \max_{|z_j| =3^j/\log T}|G(z_1, z_2, z_3, z_4)|
\end{split}
\end{equation*}   
where 
$$
G(z_1, z_2, z_3, z_4) = \sum_{n,m} \frac{a(n) a(m)}{[n,m]} B_{z_1,z_2,z_3,z_4} \Big ( \frac{n}{(n,m)} \Big ) B_{z_3,z_4,z_1,z_2} \Big ( \frac{m}{(n,m)} \Big ).   
$$
The estimate in Proposition \ref{prop2} will now follow once we establish the bound 
\begin{equation}\label{F bound}
      G(z_1, z_2, z_3, z_4)\ll (\log T_{v - 1})^{k^2-4} \Big( 18^r r! P_v^r \exp(P_v)\Big), 
\end{equation}
when $|z_j| =3^j/\log T$ for $1 \leqs j \leqs 4$.

From the multiplicative nature of the coefficients $a$, and $B_{z_1,z_2,z_3,z_4}$, we may express $G(z_1, z_2, z_3, z_4)$ as the product of 
\begin{equation} 
\label{6.3} 
\prod_{2 \leqs j < v} \Big ( \sum_{\substack{p | n,m \implies T_{j - 1} \leqs p < T_j \\ \Omega(n), \Omega(m) \leqs 500 P_j}} \frac{(k - 2)^{\Omega(n) + \Omega(m)} g(n) g(m)}{[n,m]} B_{z_1, z_2, z_3, z_4}\Big(\frac{n}{(m,n)}\Big) B_{z_3, z_4, z_1, z_2}\Big(\frac {m}{(m,n)}\Big) \Big ), 
\end{equation} 
and 
\begin{equation} 
\label{6.4} 
 \sum_{\substack{p|mn \implies T_{v - 1} \leqs p \leq T_v \\ \Omega(m) =\Omega(n)=r}} \frac{r!^2 g(m)g(n)}{[m,n]} 
 B_{z_1, z_2, z_3, z_4} \Big ( \frac{n}{(n,m)}\Big ) B_{z_3, z_4, z_1, z_2} \Big ( \frac{m}{(m,n)} \Big ). 
 \end{equation}

We now estimate the quantities in \eqref{6.3} and \eqref{6.4}.  To do this, it is helpful to note that from the definition \eqref{5.1} 
one has for $p\leqs T^{10^{-8}}$ and $|z_j| =3^j/\log T$ 
\begin{equation} 
\label{6.5} 
B_{z_1, z_2, z_3, z_4}(p^u) = \sigma_{z_1,z_2}(p^u) \Big( 1+ O\Big(\frac 1p\Big)\Big), 
\end{equation}  
from which we may deduce that 
\begin{equation} 
\label{6.51} 
|B_{z_1,z_2,z_3,z_4}(n)| \ll d_3(n) \leqs 3^{\Omega(n)},
\end{equation} 
for integers $n$ composed only of primes below $T^{10^{-8}}$, and where $d_3$ denotes the $3$--divisor function.   

Consider first the expression in \eqref{6.4}.   Using \eqref{6.51} we have $|B_{z_1,z_2,z_3,z_4}(n/(n,m))| \ll 3^r$ and $|B_{z_3,z_4,z_1,z_2}(m/(n,m))|\ll 3^r$, and 
so the quantity in \eqref{6.4} is 
\begin{align*}
&\ll 9^r \sum_{\substack{p|mn \implies T_{v - 1} \leqs p < T_v \\ \Omega(m) =\Omega(n)=r}} \frac{r!^2 g(m)g(n)}{[m,n]}\\
&  \leqs 9^r r!^2
\sum_{j=0}^{r} \sum_{\substack{p|d \implies T_{v-1} \leqs p < T_v \\ \Omega(d)=j}}  \frac{1}{d} \Big( \sum_{\substack{p|n \implies T_{v-1} \leqs p < T_v \\ \Omega(n)=r-j }} \frac{g(nd)}{n}\Big)^2.  
\end{align*} 
Since $g(nd) \leqs g(n)$, the above may be bounded by
\begin{equation} 
\label{6.6} 
\leqs 9^r r!^2 \sum_{j=0}^{r} \Big( \frac{1}{j!} P_v^j \Big) \Big( \frac{1}{(r-j)!} P_v^{r-j} \Big)^2 = 9^r r! P_v^r  \sum_{j=0}^{r} \binom{r}{j} \frac{P_v^{r-j}}{(r-j)!} \leqs 
18^r r! P_v^r \exp(P_v), 
\end{equation}  
upon noting that $\binom{r}{j} \leqs 2^r$ and $\sum_{j=0}^r P_v^{r-j}/(r-j)! \leqs \exp(P_v)$.

Now we turn to the expression in \eqref{6.3}, treating the contribution for a given $j$ in the range $2\leqs j <v$.  First we show that the constraints $\Omega(n)$ and $\Omega(m) 
\leqs 500P_j$ may be dropped from the expression there with negligible error.   
We bound these terms using Rankin's trick, in the form $\exp(\Omega(m) +\Omega(n) -500 P_j) \geqs 1$ if 
either $\Omega(m)$ or $\Omega(n)$ exceeds $500P_j$.   By \eqref{6.51} and since $|k-2| \leqs 2$, the error induced in dropping the constraint on $\Omega(m)$ and $\Omega(n)$ is 
\begin{align*}
&\leqs e^{-500 P_j} \sum_{p| m, n \implies T_{j-1} \leqs p < T_j} \frac{(2e)^{\Omega(m)+\Omega(n)}}{ [m,n]} d_3(m)d_3(n) 
\\
&\ll e^{-500 P_j} \prod_{T_{j-1} \leqs p < T_j} \Big( 1+ \frac{6e + 6e + (6e)^2}{p} +O\Big(\frac{1}{p^2} \Big) \Big) \ll e^{-100 P_j}.
\end{align*}
After discarding the constraint on $\Omega(m)$ and $\Omega(n)$, the contribution of the term in \eqref{6.3} is 
$$ 
\prod_{T_{j-1} \leqs p < T_j} \Big( \sum_{a,b=0}^{\infty} \frac{(k-2)^{a+b} g(p^{a}) g(p^b)}{p^{\max(a,b)}} B_{z_1,z_2,z_3,z_4}(p^{a-\min(a,b)}) 
B_{z_3,z_4,z_1,z_2} (p^{b-\min(a,b)})\Big).
$$  
Upon using \eqref{6.5}, we see that only the terms $a, b = 0$, or $1$ are relevant and the total contribution is 
\begin{align*}
&\prod_{T_{j-1} \leqs p < T_j} \Big( 1 + \frac{(k-2)(\sigma_{z_1,z_2}(p)+\sigma_{z_3,z_4}(p)) + (k-2)^2}{p} + O\Big(\frac{1}{p^2}\Big)\Big) \\
=& \prod_{T_{j-1} \leqs p < T_j}  \Big( 1+ \frac{k^2 -4}{p} +O\Big(\frac{1}{p^2} + \frac{\log p}{p\log T}\Big)\Big),  
\end{align*} 
since $\sigma_{z_1,z_2}(p) = p^{-z_1} + p^{-z_2} = 2 +O(\log p/\log T)$, and similarly for $\sigma_{z_3,z_4}(p)$.  
We conclude that the expression in \eqref{6.3} equals 
$$ 
\prod_{2\leqs j<v} \Big(\prod_{T_{j-1} \leqs p < T_j}  \Big( 1+ \frac{k^2 -4}{p} +O\Big(\frac{1}{p^2} + \frac{\log p}{p\log T} \Big)\Big)  +O(e^{-100 P_j})\Big) \ll (\log T_{v-1})^{k^2-4}.  
$$ 
Combining this estimate with \eqref{6.6}, the bound \eqref{F bound} follows, and with it the proof of Proposition \ref{prop2} is complete.

\end{document}